\documentclass[reqno, 11pt]{amsart}
\usepackage[top=4cm, left=3.25cm, right=3.25cm, bottom=4cm]{geometry}
\usepackage[T1]{fontenc}
\usepackage[french]{babel}
\usepackage[utf8]{inputenc}
\usepackage{amssymb, amsmath}
\usepackage[pdftex]{hyperref}

\usepackage{tikz}
\usetikzlibrary{positioning}
\usetikzlibrary{mindmap,trees}
\usepackage{verbatim}
\usepackage{float}

\usepackage{rotating}
\usepackage{color}
\usepackage{colordvi}
\usepackage{soul}

\numberwithin{equation}{section}

\theoremstyle{plain}
\newtheorem{theorem}{Th\'eorème}[section]
\newtheorem{cor}[theorem]{Corollaire}
\newtheorem{prop}[theorem]{Proposition}
\newtheorem{lemma}[theorem]{Lemme}

\newtheorem{hyp}{Hypothèse}

\newtheorem{counterhypbis}{Hypothèse}
\newenvironment{hypbis}[1]
  {\renewcommand\thecounterhypbis{#1$^\star$}\counterhypbis}
  {\endcounterhypbis}

\theoremstyle{definition}

\newtheorem{remark}[theorem]{Remarque}
\newtheorem{fact}[theorem]{Fait}
\newtheorem{definition}[theorem]{D\'efinition}

\newtheorem{notation}[theorem]{Notation}

\newcommand{\nc}{\newcommand}

\nc{\Z}{\mathbb{Z}}
\nc{\Q}{\mathbb{Q}}
\newcommand{\R}{\mathbb{R}}
\newcommand{\N}{\mathbb{N}}

\nc\LL{\mathcal L}
\nc\LLq{\LL^\mathrm{eq}}
\nc\Tq{T^\mathrm{eq}}
\nc{\Gal}{\operatorname{Gal}}
\nc{\dd}{\delta}
\nc{\sscl}[1]{\langle #1 \rangle}
\nc{\alg}[1]{#1^\textrm{alg}}
\nc{\restr}[1]{\!\!\upharpoonright_{#1}}
\newcommand{\dcl}{\operatorname{dcl}}
\newcommand{\acl}{\operatorname{acl}}
\newcommand{\pcl}{\operatorname{pcl}}
\newcommand{\tp}{\operatorname{tp}}
\newcommand{\qftp}{\operatorname{qftp}}
\nc\St{\operatorname{St}}
\nc\Stab{\operatorname{Stab}}

\nc\cb{\mathrm{Cb}}

\def\Ind#1#2{#1\setbox0=\hbox{$#1x$}\kern\wd0\hbox to
  0pt{\hss$#1\mid$\hss} \lower.9\ht0\hbox to
  0pt{\hss$#1\smile$\hss}\kern\wd0}
\def\Notind#1#2{#1\setbox0=\hbox{$#1x$}\kern\wd0\hbox to
  0pt{\mathchardef\nn="0236\hss$#1\nn$\kern1.4\wd0\hss}\hbox to
  0pt{\hss$#1\mid$\hss}\lower.9\ht0 \hbox to
  0pt{\hss$#1\smile$\hss}\kern\wd0}
\def\ind{\mathop{\mathpalette\Ind{}}}
 \def\indi#1{\mathop{\ \
    \hbox to 0pt{\hss$\mid^{\hbox to 0pt{$\scriptstyle#1$\hss}}$\hss}
    \lower4pt\hbox to 0pt{\hss$\smile$\hss}\ \ }}
\def\nindi#1{\mathop{\ \ \hbox to
    0pt{\hss$\!\not{\mid}^{\hbox to 0pt{$\scriptstyle\,#1$\hss}}$\hss}
    \lower4pt\hbox to 0pt{\hss$\smile$\hss}\ \ }}

\newcommand{\indstar}{\indi \star}

\begin{document}
\title{Un critère simple}
\date{\today}

\author{Thomas Blossier et Amador Martin-Pizarro}
\address{Universit\'e de Lyon; CNRS; Universit\'e Lyon 1; Institut
  Camille Jordan UMR5208, 43 boulevard du 11 novembre 1918, F--69622
  Villeurbanne Cedex, France.}  \address{Abteilung f\"ur Mathematische
  Logik; Mathematisches Institut; Albert-Ludwig-Universit\"at
  Freiburg; Eckerstra\ss e 1; D-79104 Freiburg; Germany.}
\email{blossier@math.univ-lyon1.fr}
\email{pizarro@math.uni-freiburg.de} 
 \thanks{Les auteurs ont conduit cette
  recherche gr\^ace au soutien du projet ValCoMo ANR-13-BS01-0006. Ils
  regrettent de ne pas encore avoir eu l'occasion d'utiliser le
  nouveau réfrigérateur mis à disposition au bâtiment Braconnier par
  le projet LABEX MILYON (ANR-10-LABX-0070) de l'Université de Lyon,
  dans le cadre du programme "Investissements d'Avenir"
  (ANR-11-IDEX-0007) de l'\'Etat Français, géré par l'Agence Nationale
  de la Recherche (ANR)}

\begin{abstract}
Nous isolons des propriétés valables dans certaines théories de
purs corps ou de corps munis d'opérateurs afin de montrer qu'une
théorie est simple lorsque les clôtures
définissables et algébriques sont contrôlées par une théorie stable
associée.
\end{abstract}

\maketitle
\thispagestyle{empty}

\section*{English Summary}

In this article, we mimic the proof of the simplicity of the
theory ACFA of generic difference fields in order to provide a
criterion, valid for
certain theories of pure fields and fields equipped with operators,
which shows that a complete theory is simple whenever its definable
and algebraic closures are controlled by an underlying stable theory.

\section*{Introduction}

Les corps différentiellement clos, les corps séparablement clos et les
corps aux différences génériques sont des exemples de corps munis
d'opérateurs qui jouent un rôle important en théorie des modèles et
ses applications à la géométrie diophantienne. Les corps aux
différences génériques, c'est-à-dire, les modèles existentiellement
clos dans la classe des corps munis d'un automorphisme, sont un cas
particulier d'un procédé général~: étant donnée une théorie stable
$T$, on considère la classe des modèles de

\[T_\sigma = T \cup \{\text{\og $\sigma$ est un automorphisme\fg}\}.\]

\noindent Quand la sous-classe des modèles existentiellement clos de
$T_\sigma$ est élémentaire, on dénote sa théorie par $TA$. Pour une
théorie stable, l'existence de $TA$ s'avère équivalente à une finitude
imaginaire uniforme, dite NFCP, et une condition technique donn\'ee
dans \cite{BS}, valable pour toute théorie fortement minimale avec
multiplicité définissable, ce qui est le cas des corps algébriquement
clos. Ainsi, la théorie ACFA des corps aux différences génériques
existe \cite{zCheHr} et est complète lorsque l'on détermine le type
d'isomorphisme de $(\acl_T(\emptyset), \sigma)$. Le corps fixe d'un
modèle de ACFA est un corps parfait pseudo-algébriquement clos avec
groupe de Galois $\hat{\Z}$. Ax \cite{jA68} montre qu'un tel corps est
un corps pseudo-fini, c'est-à-dire un corps infini modèle de la
théorie des corps finis.  Hrushovski \cite{eH12} étudie la classe des
corps parfaits pseudo-algébriquement clos avec groupe de Galois borné
et montre qu'ils satisfont une propriété d'amalgamation
généralisée. Il en déduit que la théorie est simple, une notion
introduite par Shelah comme généralisation de la
stabilité. L'amalgamation généralisée a été traitée par la suite de
façon homologique \cite{GKK13}.  \`A partir d'une configuration de
groupe dans une théorie simple, l'amalgamation généralisée permet
d'améliorer la construction dans \cite{BTW04} pour obtenir un groupe
\emph{hyperdéfinissable} \cite{KMP06}.

Cet article est une simple réécriture (nos excuses pour le mauvais jeu
de mots), sans nouvel apport, de la démonstration de la simplicité de
ACFA réalisée dans \cite{zCheHr}. Elle a pour seul but d'isoler des
propriétés, valables dans de nombreuses théories de corps munis
d'opérateurs (partie \ref{S:ex}), qui permettent de déduire dans la
partie \ref{S:simple} la simplicité lorsque les clôtures définissables
et algébriques sont contrôlées par une théorie stable. La description
de la déviation permet de donner une démonstration uniforme de
l'élimination des imaginaires (au dessus d'une sous-structure
élémentaire nommée) pour les théories de corps précédentes.

Dans la partie \ref{S:gps}, nous étudions les groupes définissables
dans ce contexte. Une propriété supplémentaire sur la clôture
définissable, satisfaite par certaines théories de corps, entraîne que
tout groupe type-définissable connexe se plonge à l’intérieur d’un
groupe algébrique. On retrouve alors le résultat connu dans le cas
différentiel ainsi qu'une version faible de \cite[Propositions 4.2 \&
4.9]{eBfD02} pour les corps séparablement clos.

Certaines théories de corps ne rentrent pas dans ce contexte,
entre-autres les paires propres de corps algébriquement clos ou les
corps séparablement clos (ou PAC bornés) de degré d’imperfection
infini. Dans ces exemples, la non-déviation est donnée en ajoutant une
condition à l’indépendance purement algébrique. En la partie
\ref{S:paires} en annexe, on montre que la démonstration du Théorème
de l’indépendance s’adapte facilement à ce cadre.

Les auteurs tiennent à remercier Zoé Chatzidakis pour ses suggestions
et remarques. Elle a joué un rôle fondamental pour la production de
cette note. Toute incohérence ou erreur sont uniquement de la
responsabilité des auteurs. Ils remercient également le rapporteur
anonyme pour ses remarques et sa lecture détaillée, qui nous ont
permis d'améliorer les premières versions de cet article.

\section{De simples théories}\label{S:ex}

\begin{definition}\label{R:KP}
Une théorie complète est \emph{simple} si dans un (tout) modèle
suffisamment saturé, il existe une relation
ternaire $\ind$ entre sous-ensembles satisfaisant les propriétés
suivantes~:\\
\begin{description}
 \item[\bf Invariance] Si $ABC \equiv A'B'C'$, alors $A\ind_C B$ si
et seulement si $A'\ind_{C'} B'$.\\
\item[\bf Symétrie] Si $A\ind_C B$, alors $B\ind_C A$.\\
\item[\bf Monotonie et Transitivité] $A\ind_C BD$ si et seulement si
$A\ind_C B$ et $A\ind_{CB} D$.\\
\item[\bf Caractère fini] $A\ind_C B$ si et seulement si $a\ind_C b $
pour tous deux sous-uples finis $a$ de $A$ et $b$ de $B$.\\
\item[\bf Caractère local] Pour tout uple fini $a$ et tout ensemble
$B$, il existe $C\subset B$ de taille bornée par $|T|$ avec $a\ind_C
B$. \\
\item[\bf Extension] Pour tous $A$, $C$ et $B$, il existe $A' \equiv_C
A$ avec $A'\ind_C B$. \\
\item[\bf Théorème de l'indépendance ($3$-amalgamation)] Si $M$ est
une sous-structure élémentaire
du modèle ambiant et $A\ind_M B$, étant
donnés
$$c\ind_M A \text { et } d\ind_M B \text{ tels que } c\equiv_M d,$$
\noindent alors il existe $e\ind_M AB$ avec $e\equiv_{MA} c$
et
$e\equiv_{MB} d$.
\end{description}
On dit que $A$ est \emph{indépendant} de $B$ sur $C$ si $A\ind_C B$.
Pour une théorie simple, la relation d'indépendance
ci-dessus est unique \cite{KP97} et coïncide avec la non-déviation.
\end{definition}

Une théorie simple est \emph{stable} si elle satisfait une version
plus forte du théorème de l'indépendance~:
\begin{description}
 \item[\bf Stationnarité] Tout type sur une sous-structure
élémentaire $M$ du modèle ambiant est \emph{stationnaire}~: étant
donnés
$$c\ind_M A \text { et } d\ind_M A \text{ tels que } c\equiv_M d,$$
\noindent alors $c\equiv_{MA} d$.
\end{description}

La théorie des corps algébriquement clos de caractéristique $p$ fixée,
avec $p$ premier ou nul, est stable~: deux sous-corps $A$ et $B$ sont
indépendants au-dessus d'un sous-corps commun $C$ si $A$ est
algébriquement indépendant de $B$ sur $C$, ou de façon équivalente, si
les corps $\alg{C(A)}$ et $\alg{C(B)}$ sont linéairement disjoints sur
$\alg C$. Le type de $A$ sur un corps parfait $C$ est stationnaire si
et seulement si l'extension de corps $C\subset C(A)$ est régulière. En
particulier, tout type sur un ensemble algébriquement clos de
paramètres est stationnaire.

L'anneau engendré par deux corps linéairement disjoints est isomorphe
à leur produit tensoriel. De façon générale, les types stationnaires
permettent de fusionner des applications élémentaires définies sur
des sous-parties indépendantes.

\begin{lemma}\label{L:Shelah}
Soient deux
sous-structures  $A$ et $B$ ayant une sous-structure  commune $C$
et des applications $C$-élémentaires $f:A\to A$ et
$g:B\to B$.  Si le type $\tp(A/C)$ est stationnaire et $\tp(A/B)$
  ne dévie pas sur $C$, alors
l'application $f\cup g$ est $C$-élémentaire.
\end{lemma}

\begin{proof}
On se place à l'intérieur d'un modèle monstre et on étend $g$ en un
automorphisme  $\tilde g$ qui fixe $C$. Alors $\tp(\tilde
g(A)/B)$ ne dévie pas sur $C$. Puisque $\tilde
g(A)\equiv_C A \equiv_C f(A)$, la stationnarité de $\tp(A/C)$ entraîne
 que $\tilde g(A)\equiv_B f(A)$. Il existe un $B$-automorphisme $h$
qui envoie $\tilde g(A)$ sur $f(A)$. La composition $h\circ \tilde g
$ étend  $f\cup g$, comme souhaité.
\end{proof}

\noindent Ce lemme permet facilement d'obtenir le résultat suivant
(cf. \cite[Fact 3.5]{dP07})~:
\begin{remark}\label{R:stat_dcl}
  Pour $x$, $y$ et $A$ dans un modèle d'une théorie stable, si le type
  $\tp(xy/A)$ est stationnaire et $y$ est définissable sur
  $x\cup\acl(A)$, alors $y$ l'est sur $x\cup A$.
\end{remark}

\begin{fact}\label{F:def}
  Dans une théorie stable $T$, si $p$ est un type stationnaire sur un
  sous-ensemble $A$ du modèle ambiant, alors son unique extension
  non-déviante $q$ à une sous-structure élémentaire $M$ est
  \emph{définissable} sur $A$~: étant donnée une formule
  $\varphi(x,y)$ à paramètres dans $A$, il existe une formule
  $\theta(y)$ à paramètres dans $A$ telle que pour tout uple $m$ dans
  $M$,
$$\varphi(x,m) \in q \text{ si et seulement si } M \models \theta(m).$$
(La formule $\theta(y)$ ne dépend que de $p$ et $\varphi$ à équivalence près.)



\noindent Enfin, toute extension non-déviante dans $T$ d'un type
défini sur une sous-structure élémentaire $M$ du modèle ambiant est
\emph{finiment satisfaisable} dans $M$~: si $a\ind_M b$, alors pour
toute formule $\varphi(x,y)$ telle que $\models\varphi(a,b)$, on a
$\models\varphi(m,b)$ pour un certain $m$ dans $M$.
\end{fact}

Une théorie du premier ordre élimine les imaginaires si, à l'intérieur
d'un modèle suffisamment saturé, chaque classe d'une relation
d'équivalence définissable sans paramètres admet l'analogue d'un corps
de définition, c'est-à-dire, un uple réel qui est fixé par tout
automorphisme si et seulement s'il fixe cette classe. Cet uple est
unique, à interdéfinissabilité près, et s'appelle le \emph{paramètre
  canonique} de la classe.  Toute théorie $T$ dans un langage $\LL$
admet une expansion $\Tq$ éliminant les imaginaires dans un langage
$\LLq$ à plusieurs sortes, une pour chaque quotient par une relation
d'équivalence définissable dans $T$ sans paramètres. La sorte
\emph{réelle}, obtenue en considérant le quotient par l'égalité,
s'identifie naturellement avec l'univers de la théorie $T$. Les
éléments des autres sortes sont dits \emph{imaginaires}. Les parties
finies peuvent être vues comme des éléments imaginaires.  La théorie
$T$ élimine les imaginaires si tout imaginaire est interdéfinissable
avec un uple réel. Elle code les ensembles finis si elle élimine les
imaginaires correspondant aux parties finies. La théorie $T$
\emph{élimine faiblement les imaginaires} si tout imaginaire $e$ est
définissable sur un uple réel $a$, qui est algébrique sur $e$. Enfin,
elle \emph{élimine géométriquement les imaginaires} si tout imaginaire
est interalgébrique avec un uple réel. Une théorie élimine les
imaginaires dès qu'elle les élimine faiblement et code les ensembles
finis.

Dans une théorie stable qui élimine les imaginaires, tout type $p$ sur
un ensemble algébriquement clos $A$ est stationnaire. De plus, il
existe un sous-ensemble $B\subset A$ définissablement clos et minimal
tel que $p$ ne dévie pas sur $B$ et $p\restr B$ est également
stationnaire. Un tel sous-ensemble est la \emph{base canonique}
$\cb(p)$ de $p$. Nous utiliserons la notation $\cb(c/D)$ pour dénoter
la base canonique du type $\tp(c/\acl(D))$.

Par la suite, fixons une théorie complète stable $T_0$ avec
élimination des quantificateurs et des imaginaires dans un langage
$\LL_0$.  L'indice $0$ fera référence à $T_0$. Ainsi, le symbole
$\indi 0$ dénote l'indépendance au sens de $T_0$, et de même pour
$\dcl_0$ ou $\acl_0$.

Dans une expansion $\LL$ du langage $\LL_0$, on considère une théorie
complète $T$ contenant $T_0^\forall$. Une $\LL_0$-structure satisfait
$T_0^\forall$ si et seulement si elle peut être plongée à l'intérieur
d'un modèle de $T_0$.

\noindent Toute $\LL$-structure est également une
$\LL_0$-structure. Ainsi, l'élimination des quantificateurs de $T_0$
entraîne que deux $0$-plongements quelconques d'un modèle $M$ de $T$
dans des modèles de $T_0$ préservent le $\LL_0$-types de toute
$\LL$-sous-structure $A$ de $M$. En particulier, au sens de $T_0$, les
clôtures définissables, respectivement algébriques, de $A$ dans chacun
des plongements sont isomorphes. On travaillera par la suite à
l'intérieur d'un modèle ambiant suffisamment saturé de $T_0$.

Nous allons lister maintenant des conditions sur les théories $T$ et
$T_0$ qui nous permettront de montrer la simplicité de $T$. Afin de
traiter également les corps séparablement clos, nous n'imposerons pas
que les structures considérées soient définissablement closes,
contrairement au cas des sous-structures géométriques \cite[Definition
2.6]{HrPi94} ou des sous-structures PAC bornées \cite[Section
2]{PP06}. Cependant, on suppose la propriété suivante~:
\begin{hyp}\label{H:dcl}
Pour tout modèle $F$ de $T$, tout élément dans $\dcl_0(F)$ est
$\LL_0$-inter\-défi\-nissable avec un uple de $F$.
\end{hyp}

Pour $T_0$ la théorie des corps algébriquement clos dans le langage
des anneaux, cette hypothèse est vérifiée si les modèles de $T$ sont
des corps, car $\dcl_0(A)$ est obtenu en ajoutant au sous-corps
engendré par $A$ les éléments purement inséparables, lorsque la
caractéristique est positive.

\begin{remark}\label{R:EIensemblefini}
  Puisque $T_0$ élimine les imaginaires, l'hypothèse $(\ref{H:dcl})$
  entraîne que toute partie $\LL_0$-définissable du modèle ambiant de
  $T_0$ qui est $F$-invariante est codée par un uple de $F$. En
  particulier, tout ensemble fini $F$-invariant est codé par un uple
  de $F$. Ainsi, la théorie $T$ code les ensembles finis.
\end{remark}

Dans un modèle suffisamment saturé $F$ de $T$, étant donnée une
sous-partie $A$ de $F$, on dénote par $\sscl A$, resp. $\dcl(A)$ ou
$\acl(A)$, la sous-structure de $F$ engendrée par $A$ au sens de $T$,
resp. la clôture définissable ou algébrique de $A$ dans $F$,
c'est-à-dire, les éléments de $F$ qui sont $\LL$-définissables ou
algébriques sur $A$.

Rappelons qu'un corps de caractéristique positive est parfait s'il
contient tous les éléments purement inséparables (à l'intérieur d'une
clôture algébrique fixée). De façon analogue, une sous-structure $A$
de $F$ est dite \emph{parfaite} si $\dcl_0(A)\cap F=A$.  Puisque la
structure $F$ est parfaite, étant donnée une sous-partie $A$, la plus
petite $\LL$-sous-structure parfaite la contenant existe, et est
appel\'ee sa \emph{clôture parfaite}, not\'ee $\pcl(A)$.

Puisque $\LL_0\subset \LL$, si $A\subset F$, alors
$\pcl(A) \subset \dcl(A)$ car $\dcl_0(A)\cap F \subset
\dcl(A)$. Notons que l'ensemble $\dcl_0(A)\cap F$ est
$\LL_0$-interdéfinissable avec $\dcl_0(A)$, par l'hypothèse
$(\ref{H:dcl})$. En revanche, même pour $A=\sscl A$, l'ensemble
$\dcl(A)$ peut être différent de $\dcl_0(A)\cap F$~: pour la plupart
des corps pseudo-finis avec $\LL_0 =\LL$ le langage des anneaux, la
clôture définissable coïncide avec la clôture algébrique \cite[Theorem
1.8]{oBzC16}.

\begin{remark}\label{R:dcl_inters}
 Si $A$ et $B$ sont deux sous-structures parfaites de $F$, alors
 $$ \dcl_0(A)\cap\dcl_0(B)=\dcl_0(A\cap B).$$
\end{remark}

\begin{proof}
  Si un élément $x$ appartient à l'intersection
  $ \dcl_0(A)\cap\dcl_0(B)$, alors l'hypothèse $(\ref{H:dcl})$
  entraîne qu'il existe un uple $a$ dans $A$ qui est
  $\LL_0$-interdéfinissable avec $x$, car $A$ est parfait. En
  particulier, l'uple $a$ appartient à $\dcl_0(B)\cap F=B$, puisque
  $B$ est parfait. On conclut que $x$ est dans $\dcl_0(A\cap B)$,
  comme souhaité.
\end{proof}

\begin{hyp}\label{H:acl}
  \begin{enumerate}
  \item[(i)] La clôture algébrique d'une sous-partie $A\subset F$
    coïncide avec la restriction à $F$ de la clôture algébrique au
    sens de $T_0$ de la sous-structure $\sscl A$ engendrée par $A$ au
    sens de la théorie $T$~:

$$\acl(A) = F \cap \acl_0(\sscl A).$$

\item[(ii)] Deux uples $a$ et $b$ de $F$ ont même type si et seulement
  s'il existe un $\LL$-isomorphisme de leurs clôtures algébriques qui
  envoie $a$ sur $b$.
\end{enumerate}
\end{hyp}

La première partie de l'hypothèse est toujours vérifiée pour une partie
 définissablement  close $A=\dcl(A)$, car $T_0$ code les
ensembles finis. Notons que toute sous-structure algébriquement close
de $F$ est parfaite.

\begin{remark}\label{R:stat}
Si $C=\acl(C)\subset F$, alors le type $\tp_0(A/C)$ d'une sous-partie
$A\subset \dcl_0(F)$ est stationnaire (cf. \cite[Lemma 2.8]{HrPi94}).
\end{remark}

\begin{proof}
Par élimination  des imaginaires de $T_0$, la base canonique
$\cb_0(A/C)$ est contenue dans $\dcl_0(A\cup C)\cap\acl_0(C) \subset
\dcl_0(F)\cap \acl_0(C)$. Par l'hypothèse $(\ref{H:dcl})$, la base
canonique est   $\LL_0$-interdéfinissable avec un uple (éventuellement
infini) de  $F\cap\acl_0(C)=C$.
\end{proof}

\begin{hyp}\label{H:qftp} Étant données deux sous-structures
  parfaites $A$ et $B$ de $F$ au-dessus d'une sous-structure commune
  algébriquement close $C$, si $A\indi 0_C B$, alors la sous-structure
  de $F$ engendrée par $A$ et $B$ est parfaite et coïncide avec la
  $\LL_0$-sous-structure engendrée par $A$ et $B$.

\noindent De plus, si $$A'\models \qftp(A/C) \text{ et } B'
\models \qftp(B/C)
 \text{ avec } A'\indi 0_C B',$$ alors
$\sscl{A'B'} \models \qftp(\sscl{AB}/C)$.
\end{hyp}

\begin{remark}\label{R:intersections}
  La première partie de cette hypothèse entraîne que
  $$\langle A,B\rangle=\langle A,B\rangle_0=\dcl_0(A,B)\cap
  F$$ \noindent et, par l'hypothèse
  $(\ref{H:acl})$, $$\acl(A,B)=\acl_0(A,B)\cap F.$$
\end{remark}

\begin{remark}\label{R:pred}
  Si $T_0$ est la théorie des corps algébriquement clos en
  caractéristique nulle, alors la théorie du corps ordonné $\R$
  satisfait les hypothèses $(\ref{H:dcl})$ et $(\ref{H:acl})$. Or,
  elle ne vérifie pas l'hypothèse $(\ref{H:qftp})$, puisque l'ordre de
  deux éléments distincts n'est pas déterminé par leurs
  relations purement algébriques.

  \noindent Le même argument montre que l'hypothèse $(\ref{H:qftp})$
  n'est pas valable pour la théorie d'un corps algébriquement clos
  muni d'un prédicat générique \cite{zChaP}.

\end{remark}

La condition suivante apparaît dans \cite[Definition 1.9]{eH02} et
joue un rôle fondamental dans la démonstration du théorème de
l'indépendance pour $T$.

\begin{hyp}\label{H:full}
Pour toute sous-structure élémentaire $N$ de $F$, on a

\[\acl_0(F) = \dcl_0(F\cup\acl_0(N)).\]
\end{hyp}

Cette condition est triviale lorsque les modèles de $T$ sont
algébriquement clos au sens de la th\'eorie $T_0$, ce qui est le cas
des corps algébriquement clos munis d'opérateurs \cite{MS14}. Si $K$
est un corps séparablement clos et $T_0$ la théorie des corps
algébriquement clos de même caractéristique, alors $\acl_0(K)$
coïncide avec $\dcl_0(K)$, qui est la clôture inséparable de $K$.

\noindent Pour une partie $A$ d'un modèle de $T_0$, nous notons pour
la suite
$$\Gal(A) = \mathrm{Aut_0}(\acl_0(A)/A),$$ le groupe des
$\LL_0$-automorphismes élémentaires de $\acl_0(A)$ fixant $A$. Lorsque
$T_0$ est la théorie d'un corps algébriquement clos, on retrouve le
groupe de Galois absolu du corps parfait $\dcl_0(A)$.

\noindent  Pour  $A \subset B$  deux parties du monstre de $T_0$ et
$\pi$  la
projection  naturelle de $\Gal(B)$ vers $\Gal(A)$, on a~:
\begin{itemize}
\item si $\tp_0(B/A)$ est stationnaire, alors $\pi$ est surjective
  (par le lemme \ref{L:Shelah})~;
\item $\acl_0(B) = \dcl_0(B\cup\acl_0(A))$ si et seulement si $\pi$
  est injective.
\end{itemize}
\begin{remark}\label{R:acl}
  Si $N$ est une sous-structure élémentaire de $F$ (ou plus
  généralement si $\acl_0(F) = \dcl_0(F\cup\acl_0(N))$ et
  $\acl_0(N)\cap F=N$), alors pour $A \supset N$ algébriquement clos
  dans $F$, on a également
\[\acl_0(A) = \dcl_0(A\cup\acl_0(N)),\]
et les projections de $\Gal(F)$ sur $\Gal(A)$ et de $\Gal(A)$ sur
$\Gal(N)$ sont des  isomorphismes.
\end{remark}
\begin{proof}
  Par la remarque \ref{R:stat} et l'hypothèse $(\ref{H:full})$, les
  applications restrictions de $\Gal(F)$ vers $\Gal(A)$ et de
  $\Gal(F)$ vers $\Gal(N)$ sont des isomorphismes. On en deduit qu'il
  en est de même pour l'application de $\Gal(A)$ vers $\Gal(N)$, donc
  $\acl_0(A) = \dcl_0(A\cup\acl_0(N))$.
\end{proof}

\begin{notation}
  Par la suite, si $N$ est une sous-structure de $F$, on posera
  $N\preceq_0 F$ pour indiquer que le $\LL_0$-réduit de $N$ est une
  sous-structure
  élémentaire du réduit  de $F$ correspondant. \\
  Lorsque $T_0$ est un réduit fortement minimal de $T$ et $N$ est une
  sous-structure algébriquement close et infinie de $F$, alors
  $N\preceq_0 F$.
\end{notation}

La démonstration de \cite[Lemma 3.18]{dP07}, simplifiée grâce à des
remarques de Zoé Chatzidakis, s'adapte facilement à notre contexte
pour donner le résultat suivant~:

\begin{lemma}\label{L:coher}
Soit $N$ une sous-structure algébriquement close contenant
une sous-structure  élémentaire de $F$ telle que $N\preceq_0 F$. Si
 $A$ et $ B$ sont deux sous-structures de $F$ contenant $N$ telles que
$$A \indi 0_N B,$$
alors le type $\tp_0(A/\acl_0(B))$ est finiment satisfaisable dans
$N$.  En particulier, le type $\tp_0(A/\acl_0(B))$ est finiment
  satisfaisable dans $\acl_0(N)$.
\end{lemma}

\begin{proof}
  Le type $tp_0(A/N)$ est stationnaire par la remarque
  \ref{R:stat}. Il suffit ainsi de montrer que $\tp_0(A/B)$ est
  finiment satisfaisable dans $N$ : en effet, ce type finiment
  satisfaisable s'étendra en un type sur $\acl_0(B)$ finiment
  satisfaisable sur $N$. En particulier, ce dernier ne dévie pas sur
  $N$ et est par stationnarité égal au type $\tp_0(A/\acl_0(B))$, car
  $A \indi 0_N \acl_0(B)$.

  \noindent Vérifions donc que $\tp_0(A/B)$ est finiment satisfaisable
  dans $N$. Soit $\varphi(x,y)$ une $\LL_0$-formule à paramètres dans
  $N$ et des uples $a$ dans $A$ et $b$ dans $B$ avec
  $\models \varphi(a,b)$. On considère l'unique extension non-déviante
  $q(y)$ de $\tp_0(b/N)$ à une $\LL_0$-sous-structure élémentaire $M$ du
  modèle ambiant de $T_0$ contenant $Na$. Par le fait \ref{F:def}, la
  $\varphi$-définition $\theta(x)$ de $q$ est définissable sur $N$.  On a
  alors $M \models \theta(a)$ car $b\indi 0_N a$.


  \noindent Par élimination des quantificateurs de $T_0$, on peut
  supposer que $\theta$ est sans quantificateurs, et donc
  $F\models \theta(a)$.  Puisque $N\preceq_0 F$, il existe une
  réalisation $n$ de $\theta$ dans $N$. En particulier
  $\models\varphi(n,b)$, comme souhaité.

\end{proof}

\begin{prop}\label{P:coher}
  Soient $A_1,\ldots,A_n$ et $B$ des sous-structures algébriquement
  closes de $F$ indépendantes au sens de $T_0$ sur une sous-structure
  algébriquement close $N$ contenant une sous-structure élémentaire
  telle que $N\preceq_0 F$ ou telle que $\acl_0(N)$ est un modèle de
  $T_0$.  Alors
  $$\sscl{\acl(A_1B),\ldots,\acl(A_nB)} \cap \acl(A_1\ldots,A_n) =
  \sscl{A_1\ldots,A_n}. $$
\end{prop}

La conclusion de la proposition est la condition $(\sharp)$ présente
dans \cite[Lemma 1.5]{KMP06}. Cette condition est fondamentale pour
avoir la $n$-amalgamation réelle au-dessus d'une sous-structure
élémentaire de $F$ (voir Définition \ref{D:nAM}).

\begin{proof}
  Puisque les structures $A_1,\ldots,A_n$ et $B$ sont indépendantes au
  sens de $T_0$ sur $N$, on a par la remarque \ref{R:intersections}
  que $\acl(A_iB) = \acl_0(A_iB) \cap F$ pour chaque $i\leq n$ et
  $\acl(A_1\ldots,A_n) = \acl_0(A_1\ldots,A_n) \cap F$. En outre,
$$\sscl{\acl(A_1B),\ldots,\acl(A_nB)}=
\dcl_0(\acl(A_1B),\ldots,\acl(A_nB))\cap F,$$
et
$$ \dcl_0(A_1,\ldots,A_n)\cap F=\sscl{A_1\ldots,A_n}.$$

\noindent Soit  $\alpha$ dans
$\dcl_0(\acl(A_1B),\ldots,\acl(A_nB)) \cap \acl(A_1\ldots,A_n)
$. Soit  $\psi(x,y_1,\ldots,y_n)$ une  $\LL_0$-formule  telle que
$\models \psi(\alpha,\zeta_1,\ldots,\zeta_n)$
pour des éléments $\zeta_i$ dans $\acl(A_iB)= \acl_0(A_iB)
\cap F$, et
$$ |\{x\,|\,\psi(x,y_1,\ldots,y_n)\}|\leq 1.$$

\noindent On trouve ainsi pour chaque $1\leq i \leq n$, une
$\LL_0$-formule $\psi_i(y_i,u),$ à paramètres dans $A_i$, chacune avec
un nombre fini de réalisations, uniforme en les paramètres, telles que
$$\models \bigwedge \psi_i(\zeta_i,b), $$
pour un certain uple $b$ dans $B$.

Comme $\alpha$ appartient à $ \acl_0(A_1,\ldots,A_n)$, il existe une
$\LL_0$-formule $\varphi(x)$ qui isole son type
$\tp_0(\alpha/\acl_0(A_1),\ldots,\acl_0(A_n))$.

\noindent Or,

$$ \models \exists x \exists y_1 \ldots \exists y_n
\left(\psi(x,y_1,\ldots,y_n) \wedge
\phi(x)\wedge \bigwedge\psi_i(y_i,b) \right).$$

\noindent Cette formule appartient donc au type
$\tp_0(b/\acl_0(A_1\ldots,A_n))$. Puisque
$$B\indi 0_N A_1,\ldots,A_n,$$

\noindent alors le type $\tp_0(b/\acl_0(A_1\ldots,A_n))$ est
finiment satisfaisable dans $\acl_0(N)$, soit par le lemme
précédent si $N\preceq_0 F$, soit par la stabilité de $T_0$ si
$\acl_0(N)$ est un modèle de $T_0$.  Ainsi, il existe $d$ dans
$\acl_0(N)$ tel que
$$\models \exists x \exists y_1 \ldots \exists y_n
\left(\psi(x,y_1,\ldots,y_n) \wedge \phi(x)\wedge
  \bigwedge\psi_i(y_i,d) \right).$$ On obtient alors une réalisation
de $\phi(x)$ dans $\dcl_0(\acl_0(A_1),\ldots,\acl_0(A_n))$.  Cette
réalisation doit être égale à $\alpha$ car $\phi$ isole son $0$-type
sur $\acl_0(A_1),\ldots,\acl_0(A_n)$.

\noindent Par l'hypothèse $(\ref{H:full})$, on en déduit que $\alpha$
appartient à $\dcl_0(A_1,\ldots,A_n,\acl_0(N))$. Puisque le type
$\tp_0(A_1,\ldots,A_n,\alpha/N)$ est stationnaire, la remarque
\ref{R:stat_dcl} entraîne que $\alpha$ appartient à
$$ \dcl_0(A_1\ldots,A_n)\cap
F=  \sscl{A_1\ldots,A_n}.$$
\end{proof}

L'hypothèse suivante, qui reflète une certaine richesse des modèles de
la théorie $T$, est satisfaite par les modèles existentiellement clos
des corps algébriquement clos munis d'opérateurs considérés dans
\cite{MS14}.

\begin{hyp}\label{H:ext}
  Soit $F'$ un autre modèle de $T$ contenant une sous-structure
  parfaite $C'$ au-dessus d'une sous-structure commune
  $A\subset F\cap F'$ parfaite dans $F$ et $F'$.

  \noindent Si $\tp_0(C'/A)$ est stationnaire, alors pour tout
  $A\subset B\subset F$, il existe une sous-structure parfaite $C$ de
  $F$ contenant $A$ qui est $\LL$-isomorphe à $C'$ au-dessus de $A$ et
  telle que $$C\indi 0_ A B.$$
\end{hyp}

Notons que si $\tp_0(C'/A)$ est stationnaire, alors
$C'\cap\acl_0(A) \subset \dcl_0(A)$. Si la théorie $T_0$ est la
théorie de purs corps algébriquement clos d'une certaine
caractéristique, et $A \subset C'$ sont des sous-corps parfaits, la
condition $C'\cap\acl_0(A) \subset \dcl_0(A)=A$
équivaut alors à ce que l'extension $C'$ sur $A$ soit régulière,
c'est-à-dire, à $C'$ et  $\acl_0(A)$ linéairement disjoints
sur $A$. De plus, cette condition entraîne  réciproquement que
le type $\tp_0(C'/A)$ est stationnaire, par
élimination des imaginaires dans $T_0$.

\begin{remark}\label{R:rel_acl}
  L'hypothèse (\ref{H:acl}) et le $\LL$-isomorphisme entre les
  structures $C'$ et $C$ ne suffisent pas pour garantir que $C$ et
  $C'$ ont le même type au sens de la théorie $T$. En revanche, si
  l'on se place au-dessus d'une sous-structure élémentaire commune
  $N$, alors cet $\LL$-isomorphisme envoie toute sous-structure
  $N \subset X' \subset C'$ qui est algébriquement close dans $F'$
  vers une sous-structure $X$ algébriquement close dans $F$~; en
  particulier $X$ et $X'$ ont alors le même type.
\end{remark}

\begin{proof}
 Par la remarque \ref{R:acl}, on a $\acl_0(X') = \dcl_0(X' \cup
\acl_0(N))$. Notons que le $\LL$-isomorphisme  entre $C'$ et $C$
induit un $\LL_0$-isomorphisme de $\acl_0(C')$ sur $\acl_0(C)$
au-dessus de $A$, puisque $\LL_0 \subset \LL$, donc on a également
$\acl_0(X) = \dcl_0(X \cup \acl_0(N))$. Soit $y \in \acl_0(X) \cap F$.
  On considère un uple $x$ dans $X$ tel que $y \in \dcl_0(x \cup
\acl_0(N))$. Par la remarque \ref{R:stat}, le type $\tp_0(xy/N)$ est
stationnaire. La remarque \ref{R:stat_dcl}  entraîne que $y$ est dans
$\dcl_0(x \cup N) \cap F\subset \dcl_0(X)\cap F$. Comme $C$ est
parfait dans $F$, et $X$ est définissablement clos au sens de
$T_0$ dans $C$, on  conclut que $y$ est dans $X$ comme souhaité.

\end{proof}

Même si nos hypothèses sont assez restrictives, plusieurs théories de
corps les satisfont.
\begin{fact}\label{F:exemples}
Les théories des~:
\begin{itemize}

  \item corps aux différences génériques ACFA en toute
caractéristique dans le langage des corps $\{+,-,\cdot,{}^{-1},0,1\}$
augmenté de   symboles de fonctions unaires pour l'automorphisme et
pour son inverse~;

   \item corps différentiellement clos, ou plus généralement corps
différentiellement clos avec $n$ dérivations qui commutent, en
caractéristique $0$~;

\item corps différence-différentiellement clos \cite{DCFA} en
    caractéristique $0$, dans le langage des corps différentiels aux
    différences~;

 \item corps munis d'opérateurs libres en caractéristique nulle
\cite{MS14}~;

\item corps séparablement clos en caractéristique positive de degré
  d'imperfection fini avec des constantes pour une $p$-base, munis des
  $\lambda$-fonctions associées~;

 \item corps aux différences génériques séparablement clos en
caractéristique positive de degré d'imperfection fini \cite{SCFA} avec
des constantes pour une $p$-base, munis des $\lambda$-fonctions
  associées et des symboles de fonctions unaires
  pour l'automorphisme et pour son inverse~;

\item corps pseudo-finis, ou plus généralement corps PAC
parfaits avec groupe de Galois borné \cite{eH02}, en
toute caractéristique dans le langage des corps augmenté de
constantes pour une sous-structure élémentaire~;

\end{itemize}
satisfont les hypothèses $(\ref{H:dcl})$, $(\ref{H:acl})$,
$(\ref{H:qftp})$, $(\ref{H:full})$  et $(\ref{H:ext})$, pour
 $T_0$ la théorie des purs corps algébriquement clos de même
caractéristique dans le langage de corps $\{+,-,\cdot,{}^{-1},0,1\}$.

\end{fact}

\begin{proof}
Rappelons que l'hypothèse $(\ref{H:dcl})$ est satisfaite dès que $T$
est une théorie de corps infinis, pour $T_0$ la théorie d'un pur corps
algébriquement clos, puisque $\dcl_0(A)$ est obtenu en ajoutant les
éléments purement inséparables au sous-corps engendré par $A$.

L'hypothèse $(\ref{H:acl})$ a été démontrée lors de l'étude
modèle-théorique de chacune des théories considérées.

Pour l'hypothèse $(\ref{H:qftp})$, notons d'abord que le corps
engendré par deux sous-corps parfaits l'est également. Dans le cas des
corps algébriquement clos, les opérateurs définis ci-dessus sont
additifs, et leur valeur sur un produit s'obtient à partir des valeurs
sur chaque facteur. Pour les corps séparablement clos munis de
$\lambda$-fonctions, toute $\LL$-structure est $0$-relativement
définissablement close et les $\lambda$-fonctions vérifient également
les propriétés ci-dessus. Enfin, les corps PAC parfaits bornés
satisfont cette hypothèse puisque $\LL=\LL_0$.

 L'hypothèse $(\ref{H:full})$ est triviale pour
tous les exemples sauf pour les corps PAC parfaits bornés, où elle
suit de \cite[Remark 1.10]{eH02}. En effet, si les groupes de
Galois  sont bornés avec même quotients finis, tout épimorphisme
continu entre eux est injectif.

\noindent Il ne reste qu'à montrer l'hypothèse $(\ref{H:ext})$. Pour
les corps algébriquement clos munis d'opérateurs, les corps $A$ et
$C'$ sont alors parfaits. Ainsi, la stationnarité du type
$\tp_0(C'/A)$ entraîne que l'extension $A\subset C'$ est régulière.
Les opérateurs s'étendent de façon canonique au corps engendr\'e par
$B$ et $C'$, car ce dernier coïncide avec le corps de fractions du
produit tensoriel $B\otimes_A C'$ si l'on place $C'$ linéairement
disjoint de $B$ sur $A$,. Puisque $F$ est existentiellement clos et
suffisamment saturé, il contient une $\LL$-copie $C$ de $C'$ sur $A$
qui est linéairement disjointe de $B$ sur $A$, comme souhaité.

Pour les corps séparablement clos de degré d'imperfection fini avec
une $p$-base nommée, le même genre d'argument permet de montrer que
l'extension de corps $F(C')$ sur $F$ est séparable de même degré
d'imperfection, ce qui permet de trouver un $\LL$-isomorphisme $\psi$
entre $C'$ et une sous-structure $C$ de $F$ linéairement indépendante
de $B$ sur $A$. De façon analogue, on obtient le résultat pour les
corps aux différences génériques séparablement clos.

Pour les corps parfaits PAC avec groupe de Galois borné, les
sous-corps $A$ et $C'$  sont parfaits. On peut supposer que $C'$ est
linéairement disjoint de $F$ sur $A$.  L'extension $F(C')$ sur $F$ est
régulière, donc la propriété PAC et la saturation donnent un
$F$-homomorphisme d'algèbres $F[B(C')]\to F$. La restriction au corps
$B(C')$ est un $B$-monomorphisme et l'image $C$ de $C'$ vérifie la
conclusion.
\end{proof}

\begin{remark} Les sous-structures bornées pseudo-algébriquement
  closes des modèles d'une théorie stable, considérées par Pillay et
  Polkowska \cite{PP06}, satisfont toutes les hypothèses
  précédentes. En effet~: en passant à une expansion conservatrice du
  langage de façon à ce que toute sous-structure soit
  $0$-définissablement close, pour garantir l'hypothèse
  $(\ref{H:qftp})$.  L'hypothèse $(\ref{H:dcl})$ correspond à ce que
  la sous-structure PAC soit définissablement close, l'hypothèse
  $(\ref{H:acl})$ suit des \cite[Propositions 3.12 \& 3.14]{dP07},
  l'hypothèse $(\ref{H:full})$ correspond à la notion d'être bornée et
  l'hypothèse $(\ref{H:ext})$ suit de la condition PAC de façon
  analogue au cas des corps PAC.

  De façon analogue, les structures munies d'une $G$-action
  existentiellement closes de Hoffmann \cite{dH17} dans le langage
  $\LL_G$, avec des symboles $\sigma_g$ et leurs inverses, pour $g$
  dans $G$, satisfont toutes les hypothèses précédentes dès que la
  clôture algébrique \emph{split}, qui correspond à l'hypothèse
  $(\ref{H:full})$ : l'hypothèse $(\ref{H:dcl})$ est vérifiée car les
  modèles sont $0$-définissablement clos par \cite[Remark 3.20]{dH17},
  l'hypothèse $(\ref{H:qftp})$ est garantie dès que l'on considère une
  extension conservatrice du langage $\LL$ pour que toute
  sous-structure soit $0$-définissablement close, l'hypothèse
  $(\ref{H:acl})$ suit de \cite[Fact 4.8 \& Lemma 4.10]{dH17}. Enfin,
  l'hypothèse $(\ref{H:ext})$ suit de la démonstration du théorème de
  l'indépendance \cite[Lemma 4.20]{dH17}, lorsque le type est
  stationnaire (sans devoir supposer que l'ensemble de base est
  algébriquement clos).
\end{remark}

\section{Une preuve simple}\label{S:simple}

Nous disposons maintenant de tous les ingrédients nécessaires pour
montrer que la stabilité de $T_0$ se transfère partiellement à la
théorie $T$. On montrera que $T$ est simple.

\begin{theorem}\label{T:simple}
  Soit une théorie stable $T_0$ avec élimination des quantificateurs
  et des imaginaires dans un langage $\LL_0$. Dans une expansion
  $\LL \supset \LL_0$, on considère une théorie complète $T$ contenant
  $T_0^\forall$.  Si le couple $(T,T_0)$ satisfait les hypothèses
  $(\ref{H:dcl})$, $(\ref{H:acl})$, $(\ref{H:qftp})$, $(\ref{H:full})$
  et $(\ref{H:ext})$, alors la théorie $T$ est simple.

\noindent Si l'on augmente le langage $\LL$ par des constantes
pour une sous-structure élémentaire de $T$, on a la caractérisation
suivante~:  $$A\ind_C B \text{ si et
  seulement si }  \sscl{AC} \indi 0_{\sscl C} \sscl{BC}.$$
\end{theorem}

\begin{proof}
  On fixe un modèle $F$ suffisamment saturée de $T$ et augmente le
  langage $\LL$ par des constantes pour une sous-structure élémentaire
  de $T$. On pose $A\ind_C B $ lorsque
  $\sscl{AC} \indi 0_{\sscl C} \sscl{BC}$ pour tous sous-ensembles
  $A$, $B$ et $C$. Nous allons vérifier que cette relation ternaire
  satisfait les propriétés de la définition \ref{R:KP}, ce qui
  entraîne la simplicité de $T$ avec les constantes nommées. Par
  \cite{eC99}, il suit que $T$ est également simple.

 Puisque $\LL$ contient $\LL_0$ et $T_0$ est stable
 avec élimination des quantificateurs, l'indépendance
ainsi définie satisfait clairement {\bf Invariance} et
{\bf Symétrie}. Le {\bf Caractère fini} est satisfait car
l'opérateur $\sscl{}$ est finitaire.

Notons que l'hypothèse $(\ref{H:acl})$ entraîne directement que
$A\ind_B \acl(B)$ et que $A\ind_C B $ si et seulement si
$\acl(AC)\ind_{\acl(C)} \acl(BC) $. On pourra donc considérer des
parties algébriquement closes pour étudier l'indépendance.

Pour {\bf Monotonie et Transitivité}, on considère des parties $A$,
$B$, $C$ et $D$ algébriquement closes de $F$ telles que $C \subset
A$ et  $C \subset B \subset D$.  Notons que si $A\ind_C B$ et
$A\ind_{B} D$, la sous-structure  $A$  est $0$-indépendante de $B$
sur $C$ et de  $D$ sur $B$. Elle est donc
$0$-indépendante de $D$ sur $C$, ce qui montre un sens de
l'implication.

\noindent Supposons maintenant que $$A \indi 0_C D.$$
\noindent Alors $A\indi 0_C B$, donc $$A\ind_C B.$$ Par
l'hypothèse $(\ref{H:qftp})$, la structure $\sscl{AB}$ est égale à la
$0$-structure engendrée par $A$ et $B$. Puisque $$A\indi
0_{B} D,$$ \noindent on conclut que $$\sscl{AB} \indi
0_{B} D,$$ donc $A\ind_{B} D$, comme souhaité.

Remarquons que pour un uple fini $a$ de $F$, la $\LL$-structure
$\sscl{a}$ a cardinalité bornée par $|T|$. De plus, si $B=\sscl B$ est
une sous-structure de $F$, alors par le caractère fini de la
$0$-indépendance, pour chaque uple fini $\xi$ de $\sscl{a}$, il existe
une sous-partie $C_\xi$ de $B$ de taille bornée par $|T_0|\leq |T|$
telle que $$\xi\indi 0_{C_\xi} B.$$ Alors si l'on pose $C_1$ la
$\LL$-sous-structure engendrée par la réunion de tous les
sous-ensembles $C_\xi$, lorsque $\xi$ parcourt $\sscl{a}$, on
a $$\sscl{a}\indi 0_{C_1} B.$$ On itère cette procédure avec
$\sscl{aC_1}$, qui a aussi cardinalité bornée par $ |T|$, pour
conclure que
$$\sscl{aC_\infty} \indi 0_{C_\infty} B,$$ où $C_\infty$ est la
$\LL$-sous-structure engendrée par la réunion de tous les $C_i$. On en
déduit le {\bf Caractère local}.

Voyons maintenant que l'indépendance satisfait {\bf
  Extension}. Considérons des parties algébriquement closes $A$, $B$
et $C$ de $F$ avec $C\subset A$. La remarque \ref{R:stat} donne que
$\tp_0(A/C)$ est stationnaire.  Par l'hypothèse $(\ref{H:ext})$, il
existe $A' \models \qftp(A/C)$ avec $A'\indi 0_C B$. Puisque $C$
contient une sous-structure élémentaire, la remarque \ref{R:rel_acl}
et l'hypothèse $(\ref{H:acl})$ entraînent que $A'\models \tp(A/C)$,
comme souhaité.

Il nous reste à montrer que l'indépendance satisfait le {\bf Théorème
de l'indépendance}. Soit  donc $N$ une sous-structure élémentaire de
$F$,  deux parties algébriquement closes $A$ et $B$ la contenant et
indépendantes sur  $N$, et deux parties algébriquement closes $C$ et
$D$ qui ont même type sur $N$ et vérifient
$$ C\ind_N A \text{\qquad et \qquad}  D\ind_N B.$$

\noindent Par {\bf Extension}, nous pouvons supposer que $D$ est
indépendante de $A$ sur $B$. \noindent Par transitivité, les
sous-structures algébriquement closes $A$, $B$ et $D$ sont
indépendantes  sur $N$. En particulier,

$$ \acl(AB) \indi 0_B \acl(BD).$$

\noindent Par l'hypothèse $(\ref{H:qftp})$, la sous-structure $E$
engendrée par $\acl(AB)$ et $\acl(BD)$ est parfaite et coïncide avec
la $\LL_0$-sous-structure engendrée par $\acl(AB)$ et $\acl(BD)$.

\noindent Vérifions maintenant que $\tp_0(E/\sscl{AD})$ est
stationnaire. Par l'hypothèse $(\ref{H:dcl})$, on peut supposer que sa
base canonique est contenue dans
$$E\cap\acl_0(AD)= E \cap \acl(AD) = \sscl{\acl(AB), \acl(BD)}  \cap
\acl(AD).$$

\noindent Par la proposition \ref{P:coher}, cette intersection est
égale à $\sscl{A D}$, ce qui donne la stationnarité du type
$\tp_0(E/\sscl{AD})$.

\noindent Puisque

$$ A \indi 0_N C \text{\qquad et
  \qquad} A \indi 0_N D,$$ l'hypothèse $(\ref{H:qftp})$ donne que
$\sscl{AD}\models \qftp(\sscl{AC}/N)$.  Cet $\LL$-isomorphisme partiel
induit un $N$-isomorphisme entre $F$ et un autre modèle suffisamment
saturé $F_1$ de $T$ tel que $\acl(AC)$ est $\LL$-isomorphe avec
$\acl_0(AD)\cap F_1$. En particulier, nous avons l'égalité
$\tp_{F_1}(D/A) = \tp_F(C/A)$.

\noindent
L'hypothèse $(\ref{H:ext})$ nous permet d'obtenir, à l'intérieur du
modèle suffisamment saturé $F_1$, une copie $E_1$ de $E$ sur
$\sscl{AD}$ qui est $0$-indépendante de $\acl_0(AD)\cap F_1$.
Par la remarque \ref{R:rel_acl}, les copies de $\acl(AB)$ et
$\acl(BD)$ dans $E_1$ sont
algébriquement closes dans $F_1$. Ainsi, la copie $B_1$ de $B$
correspondante  vérifie~:
$$ \tp_{F}(B/D) = \tp_{F_1}(B_1/D) \text{ et }  \tp_{F}(B/A) =
\tp_{F_1}(B_1/A). $$

\noindent Par saturation, il existe $B'D'$ dans $F$ tels que
$\tp_{F}(B'D'/A) = \tp_{F_1}(B_1 D/A)$.  Puisque
$\tp_F(B'/A) = \tp_{F_1}(B_1/A)= \tp_{F}(B/A)$, on peut supposer que
$B=B'$, ce qui permet de conclure.
\end{proof}

Les sous-structures PAC bornées d'une théorie superstable sont
supersimples \cite[Corollary 3.22]{dP07}.
Ce résultat est vérifié dès que les deux langages coïncident~:

\begin{cor}\label{C:ss}
  Quand $\LL=\LL_0$ (ou plus généralement, quand $\LL$ ne contient pas
  de nouveaux symboles de fonctions), si $T_0$ est superstable, alors
  $T$ est supersimple.

\end{cor}

\begin{definition}\label{D:nAM}
  Étant donné un entier $n\geq 2$, un \emph{problème complet de
    $n$-amalgamation réelle sur une sous-structure} $E$ est la donnée
  d'une famille des types $p_w(x_w)$ sur $E$ à variables réelles,
  indexée par les sous-parties propres de $\{1,\ldots,n\}$, telle
  que~:
\begin{itemize}
\item si $w\subset w'$, alors $p_{w'}(x_{w'})\upharpoonright x_w
\supset p_w$~;
\item si $a_w\models p_w$, alors
\begin{itemize}
\item $a_w\subset \acl(E,\{a_i\}_{i\in
w})$, où $a_i$ est $a_w \upharpoonright \{i\}$~;
\item l'ensemble $\{a_i\}_{i\in w}$ est $E$-indépendant.
\end{itemize}
\end{itemize}

\noindent Une théorie simple a la propriété de $n$\emph{-amalgamation
  réelle} si tout problème complet de $n$-amalgamation réelle admet un
type $p_{\{1,\ldots,n\} }$ complétant la famille et tel que les
propriétés ci-dessus sont conservées.
\end{definition}
Notons que toute théorie simple a la propriété de $3$-amalgamation
réelle sur les sous-structures élémentaires, car celle-ci est une
reformulation du théorème de l'indépendance.

En utilisant la proposition \ref{P:coher}, la démonstration ci-dessus
s'adapte facilement pour montrer le résultat suivant~:

\begin{cor}\label{C:amalg}
  Pour tout entier $n$, la théorie $T$ a la propriété de
  $n$-amalgamation réelle sur toute sous-structure élémentaire.
\end{cor}

\begin{remark}\label{R:Thm_Indep}
  Dans le théorème \ref{T:simple}, si la théorie $T_0$ est fortement
  minimale, alors le {\bf Théorème de l'indépendance} est vérifié
  au-dessus de tout ensemble algébriquement clos contenant une
  sous-structure élémentaire, car la proposition \ref{P:coher} est
  alors valable sur un tel ensemble. En particulier, toutes les
  théories de corps du fait \ref{F:exemples} satisfont le {\bf
    Théorème de l'indépendance} sur toute $\LL$-sous-structure
  algébriquement close contenant une sous-structure élémentaire. En
  revanche, on ignore si c'est le cas pour les sous-structures bornées
  pseudo-algébriquement closes d'une théorie stable quelconque
  \cite{PP06}.

\end{remark}

Une traduction directe de \cite[Proposition 4.7]{eH12} donne
l'élimination  des imaginaires de $T$ au-dessus d'une
sous-structure élémentaire nommée. On retrouve ainsi une
  démonstration  uniforme de  l'élimination des imaginaires des
  théories du fait  \ref{F:exemples}.

\begin{cor}\label{P:WEI}
  Au-dessus de constantes pour une sous-structure élémentaire, la
  théorie $T$ élimine géométriquement les imaginaires. De plus, elle
  élimine les imaginaires dès que le théorème d'indépendance est vrai
  sur toute sous-structure algébriquement close.
\end{cor}

\begin{proof}
  Vérifions d'abord que la théorie $T$ élimine géométriquement les
  imaginaires au-dessus d'une sous-structure élémentaire nommée. Soit
  $e=f(a)$ un imaginaire qui est l'image d'un uple réel $a$ par une
  fonction $f$ définissable sans paramètres. Posons $E$ la collection
  d'éléments réels de $F$ algébriques sur $e$.

  Si $a$ est algébrique sur $E$, alors on a le résultat. Sinon,
  puisque $E$ est algébrique sur $e$, il suffit de montrer que $e$ est
  algébrique sur $E$. Notons $A=\acl(E,a)$. Le lemme de Neumann
  \cite{pNeu76} entraîne qu'il existe $B$ réalisant le type
  $\tp(A/E,e)$ avec

$$A\cap B=E.$$

En particulier, l'ensemble $B$ est égal à $\acl(E,b)$ avec
$b\equiv_{E,e} a$, donc $f(b)=e=f(a)$. Construisons maintenant une
suite de sous-structures $\{A_i=\acl(E,a_i)\}_{i <\omega}$
contenant $E$, avec $A_1=A$ et $A_2=B$ telle que
$A_n A_{n+1} \equiv_E A,B$ et

$$A_{n+1} \ind_{A_n} A_1,\ldots,A_{n-1}.$$

\noindent Alors  $A_i\cap A_j=E$ pour tout $i\neq j$ et
par {\bf Transitivité}

$$ A_k \ind_{A_j} A_i$$

\noindent pour tout $i<j<k$.

Un argument à la Ramsey nous permet de supposer que la suite
$\{A_i\}_{i <\omega}$ est $E$-in\-dis\-cer\-nable. L'élimination de
quantificateurs de $T_0$ entraîne que cette suite est également
$\LL_0$-indiscernable sur $E$. La stabilité de $T_0$ donne qu'elle est
totalement $\LL_0$-indiscernable.

La remarque \ref{R:stat} implique que le type $\tp_0(A_3/A_1)$ est
stationnaire et de même pour $\tp_0(A_3/A_2)$. Puisque $A_3\ind_{A_2}
A_1$, la caractérisation de l'indépendance entraîne que $$A_3\indi
0_{A_2} A_1.$$ Comme $A_1A_2 A_3$ a même $T_0$-type que $A_2 A_1
A_3$, on conclut que $A_3\indi 0_{A_1} A_2$. La base canonique
$\cb_0(A_3/A_1A_2)$ est  contenue dans
$\dcl_0(A_1)\cap\dcl_0(A_2)=\dcl_0(E)$, par la remarque
\ref{R:dcl_inters}. Ainsi

$$A_3 \indi 0_E A_1A_2,$$

\noindent ce qui entraîne que $A_3$ et $A_1$ sont $T$-indépendants
sur $E$. En particulier, l'imaginaire $e$ est  algébrique sur $E$, ce
qui donne l'élimination géométrique des imaginaires.

Supposons maintenant que la théorie $T$ vérifie le {\bf Théorème de
  l'indépendance} sur toute structure algébriquement close. Pour
  montrer  que $T$ élimine les imaginaires au-dessus de la
  sous-structure élémentaire  nommée, il
  suffit de montrer qu'elle les élimine  faiblement par la remarque
  \ref{R:EIensemblefini}.  L'indépendance

$$A_3 \indi 0_E A_1,$$

\noindent donne deux réalisations $\tilde a$ et $c$ du type $\tp(a/E)$
indépendantes sur $E$ telles que $f(c)=f(\tilde a)=f(a)=e$. On peut
supposer que $c$ est indépendant de $a$ sur $E$. Pour montrer que $e$
est définissable sur $E$, il suffit de montrer que la valeur de $f$
est constante sur toute autre réalisation du type $\tp(a/E)$. Sinon,
soit $d$ une réalisation avec $f(d)\neq f(a)$. Si $\tilde d$ réalise
le type $\tp(d/E,c)$ et est indépendante de $a$, la {\bf Transitivité}
donne que $\tilde d \ind_E a$ et $f(\tilde d)\neq f(c)=f(a)$. On peut
supposer que $d$ est indépendante de $a$.  Les types $\tp(a/E,c)$ et
$\tp(d/E,a)$ sont deux extensions non-déviantes du type $\tp(a/E)$ et
on a $c\ind_E a$. Le {\bf Théorème de l'indépendance} appliqué au
diagramme

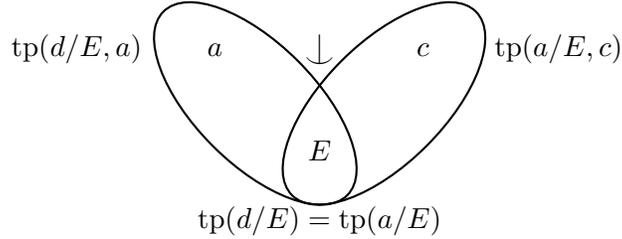
{\begin{figure}[h]
   \centering
   \begin{tikzpicture}[scale=.8, text height=1ex, text depth=1ex]
 \def\Pi{(-4.2,0) ellipse (6 and 1)}
\def\Pii{(4.2,0) ellipse (6 and 1)}
\begin{scope}[rotate=-45, xscale=0.36, yscale=1]
        \path[draw,thick] \Pi;
 \end{scope}
 \begin{scope}[rotate=45, xscale=0.36, yscale=1]
    \path[draw,thick] \Pii;
     \end{scope}

    \node (0,-5)  (base) {};
    \node[above=-5mm of base] (E) {$E$};
\node[above=7.8mm of E] (ind) {$\ind$};
\node[left=8mm of ind] (a) {$a$};
\node[ right =8mm of ind] (c) {$c$};
\node[left=6mm of a] (P2) {$\tp(d/E,a)$};
\node[right =6mm of c] (P1) {$\tp(a/E,c)$};
\node[below=2mm of base] (P) {$\tp(d/E)=\tp(a/E)$};
 \end{tikzpicture}
 \end{figure}}

\noindent donne une réalisation $\zeta$
indépendante de $c,a$ sur $E$ avec

$$ \zeta \models \tp(d/E,a) \text{ et } \zeta \models \tp(a/E,c).$$

\noindent En particulier $f(\zeta)\neq f(a)=f(c)= f(\zeta)$, ce qui
est une contradiction.

\end{proof}

\begin{remark}\label{R:WEI_general}
  Notons que la démonstration ci-dessus s'appuie uniquement sur le
  fait que $T$ est une théorie simple contenant la partie universelle
  d'une théorie stable $T_0$ qui élimine les quantificateurs et les
  imaginaires dans un sous-langage $\LL_0$ telle que~:
\begin{itemize}
 \item la théorie $T$ satisfait  les hypothèses $(\ref{H:dcl})$ et
$(\ref{H:acl})$~;
\item pour des sous-structures algébriquement closes $A$ et $B$ d'un
modèle suffisamment saturé $F$ de $T$,
$$A\ind_{A\cap B} B \text{ si et seulement si } A\indi 0_{A\cap B}
B~;$$
\item la théorie $T$ vérifie le {\bf Théorème de
  l'Indépendance} sur toute structure algébriquement close.
\end{itemize}

En particulier, la théorie d'un corps algébriquement clos muni d'un
prédicat générique élimine les imaginaires.
\end{remark}

Certains corps munis d'opérateurs considérés dans le fait
\ref{F:exemples} sont stables, notamment les corps différentiellement
clos et les corps séparablement clos de degré d'imperfection
fini. Dans ces deux cas, le $\LL$-type d'un uple $a$ est déterminé par
le $\LL$-type sans quantificateurs de sa clôture définissable.

\begin{cor}\label{C:stable}
  Supposons que deux uples $a$ et $b$ dans $T$ ont même type dès qu'il
  existe un $\LL$-isomorphisme de leurs clôtures parfaites qui envoie
  $a$ sur $b$.  La théorie $T$ est alors stable.
\end{cor}

\begin{proof}
Il suffit de montrer que tout type sur un ensemble algébriquement clos
contenant une sous-structure élémentaire est stationnaire, ce qui
suit directement de l'hypothèse $(\ref{H:qftp})$ si $T$ satisfait la
propriété ci-dessus.
\end{proof}

\section{Groupes et plongements}\label{S:gps}

Tout groupe définissable dans un corps différentiellement clos se
plonge comme sous-groupe d'un groupe algébrique, ce qui a été montré
par Pillay \cite{aP97} pour répondre à une question de Kolchin. Cette
caractérisation joue un rôle important dans la théorie de Galois
différentielle des extensions de Picard-Vessiot \cite{aP98}. Kowalski
et Pillay ont obtenu un résultat similaire \cite{pKaP02}, modulo un
noyau fini, pour tout groupe connexe définissable dans un corps aux
différences générique. Il existe des résultats analogues pour décrire
les groupes définissables dans de nombreuses théories de corps, purs
ou munis d'opérateurs, par exemple dans un corps pseudo-fini
\cite{HrPi94} ou dans un corps séparablement clos \cite{mM94, eBfD02}.
Dans ce dernier cas, tout groupe définissable dans un corps
séparablement clos $L$ est définissablement isomorphe aux points
$L$-rationnels d'un groupe algébrique défini sur $L$.

Les démonstrations pour chacun de ces exemples s'appuient de façon
essentielle sur le théorème de configuration de groupe, dû à
Hrushovski \cite{eH90}, qui est une généralisation abstraite du
résultat de Weil \cite{aW55} produisant un groupe algébrique à partir
d'une loi de groupe birationnelle. La construction de Hrushovski donne
un groupe connexe constitué des germes de fonctions génériquement
définies.  Même dans le cas de la théorie des corps algébriquement
clos, le groupe des germes ainsi obtenu n'est pas forcément un groupe
algébrique, mais pro-algébrique~: une limite inverse de groupes
définissables.

Soit $G$ un groupe type-définissable dans une théorie simple. Un
élément $g$ dans $G$ est \emph{générique (à gauche)} sur l'ensemble
des paramètres $A$ si $$h\cdot g \ind A,h$$
\noindent dès que
$$h\ind_A g.$$
\noindent Un type est générique (à gauche) si toutes ses réalisations
le sont. Les types génériques existent~: ils coïncident avec les types
dans $G$ de rang stratifié local maximal.  En particulier, un type est
générique à droite si et seulement s'il l'est à gauche. Toute
extension ou restriction non-déviante d'un type générique l'est
aussi. L'inverse d'un générique est également générique. Enfin, le
produit de deux génériques indépendants l'est aussi et est indépendant
de chaque facteur. En outre, tout élément de $G$ peut s'écrire comme
le produit de deux génériques, pas forcément indépendants entre eux.

Étant donné un ensemble de paramètres $A$, la \emph{composante
  connexe} $G_A^0$ de $G$ sur $A$ est l'intersection de tous les
sous-groupes $G$ d'indice borné type-définissables sur $A$. Le
sous-groupe $G_A^0$ est distingué et type-définissable sur $A$
d'indice borné. Le groupe $G$ est \emph{connexe sur} $A$ si
$G=G_A^0$. Si la théorie est stable ou supersimple, alors la
composante connexe sur $A$ d'un groupe type-définissable est
l'intersection des sous-groupes définissables d'indice fini.

\noindent Pour une théorie stable, la composante connexe ne dépend pas
de l'ensemble de paramètres choisi et elle est définissable sur le
même ensemble de paramètres que $G$. De plus, chaque classe de $G^0$
contient un seul type générique et $G$ agit par translation sur
l'ensemble des types génériques, qui sont en correspondance avec
$G/G^0$. Dans la cas d'une théorie simple, la composante connexe peut
contenir plusieurs types génériques. En revanche, on a un résultat
partiel \cite[Proposition 2.2]{PSW98}.

\begin{prop}\label{P:prod_gen}
  Au-dessus d'une sous-structure élémentaire $N$, étant donnés trois
  types génériques $p$, $q$ et $r$ de $G_N^0$, alors il existe deux
  réalisations $g\models p$ et $h\models q$ indépendantes sur $N$
  telles que $g\cdot h\models r$.
\end{prop}

Ziegler \cite[Theorem 1]{mZ04} montre une sorte de  contraposée à la
proposition précédente, qui se généralise facilement au cas d'une
théorie simple \cite[lemme  1.2]{BMPW14}. Elle permet de déduire le
résultat suivant~:

\begin{lemma}\label{L:Iso}
  À l'intérieur d'un modèle d'une théorie simple, soient $G_1$ et
  $G_2$ deux groupes types-définissables sur une sous-structure
  élémentaire $N$. Étant donnés des éléments $a_1$ et $b_1$ de $G_1$,
  $a_2$ et $b_2$ de $G_2$ tels que~:
\begin{enumerate}
\item\label{I:interalg} les éléments $a_1$ et $a_2$, sont
  interdéfinissables sur $N$, et de même pour $b_1$ et $b_2$, ainsi
  que pour $a_1\cdot b_1$ et $a_2\cdot b_2$~;
\item les éléments $a_1$, $b_1$ et $a_1\cdot b_1$ sont indépendants
  deux-à-deux sur $N$~;
\end{enumerate}

\noindent alors, l'élément $a_1$, resp.\ $a_2$, est générique dans un
translaté type-définissable sur $N$ d'un sous-groupe $H_1$
type-définissable sur $N$ de $G_1$, resp.\ $H_2$ de $G_2$.
En outre, les sous-groupes $H_1$ et
$H_2$ sont connexes et définissablement isomorphes sur $N$ .
\end{lemma}

\noindent Notons que les sous-groupes $H_1$ et $H_2$ sont uniques, à
commensurabilité près. Ce lemme s'adapte facilement au cas des groupes
$*$-définissables.

\begin{remark}\label{R:Iso}
  Si nous avons des éléments seulement interalgébriques dans le lemme
  précédent, alors les groupes obtenus sont définissablement
  isogènes~: une correspondance à fibres finies entre deux
  sous-groupes d'indice borné.
\end{remark}

Par la suite, fixons une théorie $T$ satisfaisant les hypothèses
$(\ref{H:dcl})$, $(\ref{H:acl})$, $(\ref{H:qftp})$,  $(\ref{H:full})$
et $(\ref{H:ext})$ par  rapport à une théorie stable  $T_0$ avec
élimination des quantificateurs et des imaginaires. On travaille à
l'intérieur d'un modèle $F$ de $T$ suffisamment saturé.

\begin{prop}\label{P:isog}

  Soit $G$ un groupe type-définissable sur une sous-structure
  élémentaire $|T|^+$-saturée $N$ de $F$.

\begin{itemize}
\item Si $F$ est $\LL_0$-définissablement clos, alors $G$ est isogène
  à un sous-groupe de $H(F)$, avec $H$ un groupe $\LL_0$-définissable
  sur $N$.
\item Si $T_0$ est la théorie d'un corps algébriquement clos de
  caractéristique fixée, alors $G$ est isogène à un sous-groupe de
  $H(F)$, avec $H$ un groupe algébrique sur $N$.
\item   Si la théorie $T$ est stable ou supersimple et
$T_0$ est la théorie d'un  corps  algébriquement clos de caractéristique
  fixée, alors tout groupe définissable admet un sous-groupe
  définissable  d'indice fini qui s'envoie par un morphisme
  définissable à noyau fini dans le sous-groupe des points
  $F$-rationnels d'un groupe algébrique défini sur $N$.
\end{itemize}

\end{prop}

\begin{proof}

  \noindent En prenant la composante connexe sur $N$ de $G$, on peut
  supposer que $G$ est connexe sur $N$. Pour deux génériques
  indépendants $a$ et $b$ du groupe $G$ sur $N$, les hypothèses
  (\ref{H:acl}) et (\ref{H:qftp}) donnent l'inclusion suivante~:
$$ \acl(a\cdot b, N) \subset \acl_0(\acl(a, N)\cup\acl(b, N)).$$

\noindent En prenant trois génériques $a$, $b$ et $c$ de $G$
indépendants sur $N$, on obtient la $T_0$-configuration de groupe
suivante~:

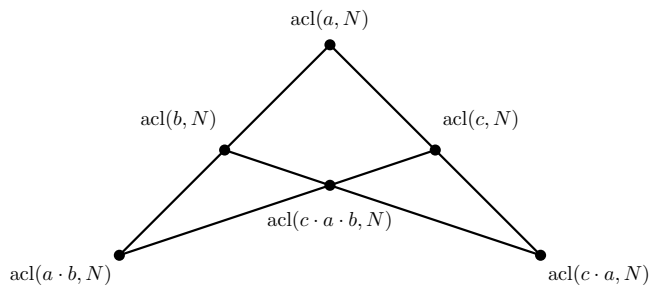
\begin{figure}[H]
  \centering
\scalebox{.7}{
  \begin{tikzpicture}[thick,scale=1, every
node/.style={scale=1},text height=2ex,text depth=2ex]

  \draw (0,4) node [above = -1mm] {$\acl(a, N)$} ;
     \fill   (0,4)   circle (3pt);

\draw (2,2) node [above right] {$\acl(c, N)$} ;
\fill (2,2) circle
(3pt);

 \draw (-2,2) node [above left] {$\acl(b, N)$} ;
 \fill   (-2,2)  circle (3pt);
 \draw (-4,0) node [below left] {$\acl(a\cdot b, N)$} ;
      \fill  (-4,0)  circle (3pt);
       \draw     (4,0) node [below right ] {$\acl(c\cdot a, N)$} ;
       \fill   (4,0)  circle (3pt);
   \draw     (0,1.33) node [below = 3mm ] {$\acl(c\cdot a\cdot b, N)$} ;
      	\fill  (0,1.33)  circle
(3pt);

      \draw[very thick]  (0,4) --  (2,2) --  (4,0) -- (0,1.33)
--
(-2,2) ;
      \draw[very thick] (0,4) -- (-2,2) -- (-4,0) -- (0,1.33)
-- (2,2) ;
  \end{tikzpicture}
  }
\end{figure}

\noindent À partir de cette configuration, on peut construire de façon
analogue à \cite{HrPi94} un groupe connexe $\LL_0$-$*$-définissable
$H^*$ sur $N$ et deux génériques $h^*$ et $h'^*$ indépendants de $H^*$
dans $\dcl_0(F)$ tels que les éléments $\acl(a,N)$ et $h^*$ sont
$\LL_0$-interalgébriques sur $N$, et de même pour $\acl(b,N)$ et
$h'^*$, ainsi que pour $\acl(a\cdot b,N)$ et $h^*\cdot h'^*$. Par
stabilité et élimination des imaginaires de $T_0$, on peut supposer
que le groupe $H^*$ est une limite inverse de groupes
$\LL_0$-définissables sur $N$. Il existe ainsi un groupe $H$ de cette
limite inverse et deux génériques $h$ et $h'$ indépendants de $H$ tel
que $a \in \acl_0(h,N)$ et $h \in \acl_0(\acl(a,N)) \cap \dcl_0(F)$
(et de même pour $b$ relativement à $h'$, et $a\cdot b$ relativement à
$h\cdot h'$).

Quand $F$ n'est pas $\LL_0$-définissablement clos, on ne sait pas de
façon générale si les génériques $h$, $h'$ et $h\cdot h'$ tombent dans
$H(F)$ ni si ce dernier ensemble forme un sous-groupe. Par contre, si
$F$ est $\LL_0$-définissablement clos, la remarque \ref{R:Iso}
entraîne que $G$ est isogène à un sous-groupe de $H(F)$. Ce
sous-groupe peut être propre car le $\LL_0$-générique $h$ n'est pas
forcement un générique de $H$ au sens de $T$.

Si $T_0$ est la la théorie d'un corps algébriquement clos de
caractéristique positive fixée, on peut supposer que $H$ est également
connexe sur $N$. Par {\bf Caractère fini}, il existe une
sous-structure élémentaire $N_0$ de $N$ de taille $|T|$ contenant tous
les paramètres témoignant les interalgébricités ci-dessus et les
paramètres nécessaires pour définir $H$.  Par saturation de $N$, il
existe une suite $\{h_i\}_{i \in \N}$ de $\LL_0$-génériques
indépendants de $H$ sur $N_0$ qui appartiennent au corps
$k=\dcl_0(F)\cap\acl_0(N)$.

\noindent Par la remarque \ref{R:stat}, le $\LL_0$-type d'éléments de
$\dcl_0(F)$ sur $N_0$ est stationnaire. L'indépendance algébrique pour
un $\LL_0$-type stationnaire est équivalente au fait que les corps
engendrés soient linéairement disjoints. Ainsi, une adaptation directe
de la démonstration du \cite[Theorem 1-B]{eB89} entraîne l'existence
d'un corps $k_0$ finiment engendré sur $N_0$ contenu dans $k$ et d'une
bijection $\LL_0$-définissable $f$ sur $k_0$ tels que
$f(h^{-1}) \in k_0(f(h))$ et $f(h\cdot h')$ appartient à
$k_0(f(h),f(h'))$ (où $h^{-1}$ est l'inverse de $h$ dans le groupe
$H$). En composant avec une certaine puissance du Frobenius, on peut
supposer que $k_0 \subset N$ et que $f(h)$ et $f(h')$ sont dans $F$.
Par le théorème de Weil \cite[Theorem 4.1]{eBfD02}, on obtient un
groupe algébrique $\tilde H$ défini sur $k_0$ et deux génériques
$\tilde h$ et $\tilde h'$ indépendants sur $N$ tel que
$a \in \acl_0(\tilde h,N)$ et
$\tilde h \in \acl_0(\acl(a,N)) \cap \dcl_0(F)$ (et de même pour $b$
relativement à $\tilde h'$, et $a\cdot b$ relativement à
$\tilde h\cdot \tilde h'$). La remarque \ref{R:Iso} entraîne à nouveau
que $G$ est isogène à un sous-groupe de $\tilde H(F)$.

Si $T$ est stable ou supersimple, la composante connexe sur $N$ d'un
groupe définissable est l'intersection des sous-groupes
$N$-définissables d'indice fini. Un argument de compacité entraîne
l'existence d'un sous-groupe définissable de $G$ d'indice fini qui se
plonge à noyau fini dans les points $F$-rationnels $H(F)$ d'un groupe
$\LL_0$-définissable. On en déduit le résultat souhaité.
\end{proof}

En supposant un certain contrôle sur la clôture
définissable, nous allons montrer comment le théorème de
Weil-Hrushovski permet de plonger définissablement tout groupe
type-définissable connexe à l'intérieur d'un groupe
$L_0$-définissable. On retrouve alors les résultats connus dans le cas
différentiel et pour les $\sigma$-groupes dont le produit est donné
par des fonctions $\sigma$-rationnelles ainsi qu'une version faible du
cas des corps séparablement clos.

\begin{prop}\label{P:WH}
  Soit $G$ un groupe connexe type-définissable à paramètres dans une
  sous-structure élémentaire $|T|^+$-saturée $N$ de $F$ tel que,
  lorsque $a$ et $b$ sont deux génériques indépendants de $G$ sur $N$,
  on a
\[\label{E:cond}
\dcl(a\cdot b,N) \subset \dcl_0(\dcl(a,N)\cup\dcl(b,N)). \tag{$\natural$}
\]
Alors, si $F$ est $\LL_0$-définissablement clos ou si $T_0$ est la
théorie des corps algébriquement clos, il existe un groupe
$\LL_0$-définissable $H$ à paramètres dans $N$ tel que le groupe $G$
se plonge définissablement dans le sous-groupe $H(F)$ constitué des
points $F$-rationnels.

\noindent Lorsque $T_0$ est la théorie d'un corps algébriquement clos
de caractéristique fixée, on peut supposer que le groupe $H$ est un
groupe algébrique.
\end{prop}

\begin{proof}

  Pour la preuve, on travaille au-dessus de la
  sous-structure élémentaire $N$.  Soient $a$ et $b$ deux
  réalisations indépendantes du même type générique $p$ de $G$ sur $N$.

  Vérifions d'abord que l'on peut supposer qu'il existe une
  réalisation $c$ de $p$ indépendante de $a,b$ telle que chacun des
  éléments dans $\{a,b, a\cdot b, c, b\cdot c, a\cdot b\cdot c\}$
  réalise le type $p$.

  \noindent Par la proposition \ref{P:prod_gen}, il existe
  $x_1\models p$ indépendant de $a, b$ avec $b\cdot x_1 \models p$. De
  même, il existe $x_2\models p$ indépendant de $a, b$ avec
  $ a\cdot b \cdot x_2 \models p$. Le {\bf Théorème de l'indépendance}
  appliqué au diagramme suivant entraîne l'existence d'un élément $c$
  comme souhaité.  \vskip1cm
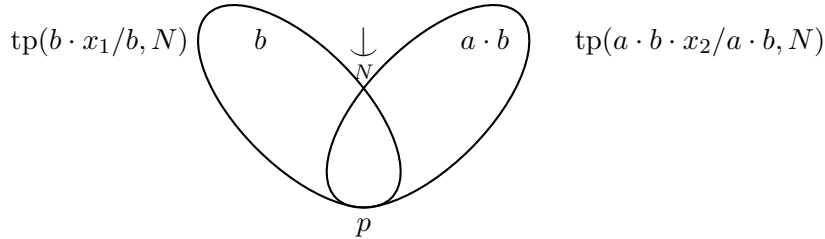
\begin{figure}[h]
   \centering
   \begin{tikzpicture}[scale=.8, text height=1ex, text depth=1ex]
 \def\Pi{(-4.2,0) ellipse (6 and 1)}
\def\Pii{(4.2,0) ellipse (6 and 1)}
\begin{scope}[rotate=-45, xscale=0.36, yscale=1]
        \path[draw,thick] \Pi;
 \end{scope}
 \begin{scope}[rotate=45, xscale=0.36, yscale=1]
    \path[draw,thick] \Pii;
     \end{scope}

    \node (0,-5)  (base) {};
    \node[above=-5mm of base] (E) {};
\node[above=9.2mm of E] (ind) {$\ind\limits_N$};
\node[left=8mm of ind] (a) {$b$};
\node[ right =8mm of ind] (c) {$a\cdot b$};
\node[left=6mm of a] (P2) {$\tp(b\cdot x_1/b, N)$};
\node[right =6mm of c] (P1) {$\tp(a\cdot b\cdot x_2/a\cdot b, N)$};
\node[below=2mm of base] (P) {$p$};
 \end{tikzpicture}
 \end{figure}

 Posons $a^*=\dcl(a, N)$, et de même pour $b$, $c$ et les produits
 correspondants. Notons que $a^*$, $b^*$, $(a\cdot b)^*$, $c^*$,
 $(c\cdot a)^*$ et $(c\cdot a\cdot b)^*$ sont des réalisations du même
 $T_0$-type sur $N$, que l'on note $p_0$. Par la remarque
 \ref{R:stat}, le type $p_0$ est stationnaire. Par la condition
 \eqref{E:cond}, il existe une fonction $0$-définissable $f$ définie
 sur $p_0^{\otimes 2}$ avec $(a\cdot b)^*=f(a^*,b^*)$.  La
 caractérisation de l'indépendance donne que l'uple $(a^*, b^*, c^*)$
 réalise $p_0^{\otimes 3}$ et

$$ f(a^*, f(b^*,c^*))=(a\cdot b\cdot c)^*= f(f(a^*,b^*),c^*).$$

La démonstration du théorème de groupes génériquement présentés de
Hrushovski \cite[Theorem 1]{eH90} adaptée aux uples infinis
(cf. \cite[Theorem 4.1]{pKaP02}) donne un groupe
$\LL_0$-$*$-défi\-nissable $H^*$ sur $N$ tel que son générique $h^*$
est $\LL_0$-interdéfinissable avec $a^*$.

De même que précédemment, il existe alors un groupe
$\LL_0$-définissable $H$ et deux génériques $h$ et $h'$ indépendants
de $H$ tel que $a \in \dcl_0(h,N)$ et $h \in \dcl_0(a^*)$ (et de même
pour $b$ relativement à $h'$, et $a\cdot b$ relativement à
$h\cdot h'$).  Si $F$ est $\LL_0$-définissablement clos, le lemme
\ref{L:Iso} entraîne que $G$ se plonge dans un sous-groupe de
$H(F)$. Si $T_0$ est la théorie d'un corps algébriquement clos de
caractéristique positive, on peut supposer comme auparavant que $H$
est un groupe algébrique défini sur $N$, tel que $h$, $h'$ et
$h\cdot h'$ sont dans $F$.

\end{proof}

\begin{remark}\label{R:Sylvain}
  Si $G$ n'est pas connexe et la théorie $T$ est stable ou
  supersimple, le même argument que celui de la fin de la
  démonstration de la proposition \ref{P:isog} entraîne l'existence
  d'un sous-groupe définissable de $G$ d'indice fini qui se plonge
  dans les points $F$-rationnels $H(F)$ du groupe $\LL_0$-définissable
  $H$. Ceci permet de plonger sans peine $G$ définissablement dans
  $H_1(F)$, pour un groupe $H_1$ aussi $\LL_0$-définissable (cf. la
  démonstration de \cite[Proposition 4.9]{eBfD02}).

  Dans un contexte qui nous semble similaire, l'existence de types
  génériques définissables permet d'étudier les groupes définissables
  dans certains corps valués enrichis \cite{sRprep}.

\end{remark}

\section{Annexe : Théorème de l'indépendance, le
  retour}\label{S:paires}

L'approche des parties précédentes permet de montrer la simplicité de
certaines théories à partir de la stabilité d'une théorie ambiante
$T_0$, en vérifiant que l'indépendance dans $T_0$ induit directement
la non-déviation dans ces nouvelles théories. Certaines théories
stables (ou simples) ne rentrent pas dans ce contexte, entre-autres
les paires propres de corps algébriquement clos ou les corps
séparablement clos (ou PAC bornés) de degré d'imperfection
infini. Dans ces exemples, la non-déviation est donnée en ajoutant une
condition à l'indépendance algébrique.  Cependant, nous allons voir
dans cette annexe que la démonstration du {\bf Théorème de
  l'indépendance} s'adapte facilement à ce cadre. De plus, l'étude des
groupes précédente s'étend ici également sans difficulté.

On utilisera les mêmes notations que précédemment pour les théories
$T$ et $T_0$. Afin de considérer les exemples ci-dessus, la théorie
$T$ est remplacée par le diagramme élémentaire d'un modèle de $T$ (et
on augmente $\LL_0$ et $\LL$ par les constantes qui énumèrent ce
modèle).

Deux $\LL$-sous-structures $A$ et $B$ d'un modèle suffisamment saturé
$F$ de $T$ telles que leur intersection $C$ est algébriquement close,
sont \emph{$\star$-indépendantes} au-dessus de $C$ si
$$ A \indi 0_C B$$ et la $\LL$-sous-structure $\sscl{A,B}$ est
parfaite et coïncide avec $\sscl{A,B}_0 \cap F$, où $\sscl{A,B}_0$
dénote la $\LL_0$-sous-structure de $F$ engendrée par $A, B$. Plus
généralement, deux parties $X$ et $Y$ de $F$ sont
\emph{$\star$-indépendantes} au-dessus d'une structure algébriquement
close $C$ si $\acl(X,C)$ et $\acl(Y,C)$ le sont.  Dans ce cas, on note
$X\indstar_C Y$.

Dans le cas des corps séparablement clos (ou plus généralement PAC
bornés) de degré d'imperfection $e$ (avec $e$ dans $\N\cup\{\infty\}$)
et munis des $\lambda$-fonctions, la propriété
$\sscl{A,B} = \sscl{A,B}_0 \cap F$ équivaut à ce que l'extension $F$
sur le sous-corps engendré par $A$ et $B$ soit séparable,
c'est-à-dire, que $F^p$ et le sous-corps $A\cdot B$ engendré par $A$
et $B$ soient linéairement disjoints sur $(A\cdot B)^p$.  En degré
d'imperfection fini, cette propriété est toujours vérifiée (car l'on a
ajout\'e au langage des symboles de constantes pour une $p$-base).  De
la même façon, pour les paires $(K,k)$ de corps algébriquement clos
dans le langage de Delon \cite{fD12}, la propriété précédente équivaut
au fait que $k$ et le corps $A\cdot B$ engendré par $A$ et $B$ soient
linéairement disjoints au-dessus de leur intersection
$k\cap (A\cdot B)$.

Pour la théorie des automorphismes génériques de corps séparablement
clos de degré d'imperfection infini~\cite{SCFA}, la
$\star$-indépendance dans le langage $\LL_{\sigma,\lambda}$ des corps
aux différences munis des fonctions $\lambda$ correspond à la
non-déviation~\cite[Corollary 3.8]{SCFA}. Ceci est également le cas
pour la modèle compagne ${\rm SCFE}_{p,e}$ des corps $K$ séparablement
clos de caractéristique $p$ et de degré d'imperfection $e$ avec un
endomorphisme $\sigma$ non surjectif, vérifiant que
$\sigma(K)^{alg}\subset K$, dans le langage $\LL_{\sigma,\lambda}$
complété par la fonction partielle $\sigma^{-1}$ et les fonctions du
langage de Delon pour la paire $(K,\sigma(K))$~\cite[Corollary
4.16]{SCFE}.

Dans ces exemples, on vérifie directement que la $\star$-indépendance
satisfait la {\bf Transitivité}, les {\bf Caractères fini et local} et
l'{\bf Extension}. Nous allons isoler des conditions, qui nous
semblent naturelles, permettant de vérifier le {\bf Théorème de
  l'indépendance} de manière unifiée. Ainsi, on retrouve le fait que
la non-déviation correspond à la description ci-dessus pour les
exemples précédents.

Puisque les $\lambda$-fonctions relatives sur une somme ou un produit
sont des expressions polynomiales sur les $\lambda$-fonctions de
chaque facteur, on vérifie facilement que les exemples ci-dessus
satisfont l'hypothèse suivante~:
\begin{hypbis}{3}\label{H:qftp_star}
Si $A\indstar_C B$ et $A'\indstar_C B$, avec $A'\models
\qftp(A/C)$, alors $$\sscl{A',B}\models \qftp(\sscl{A,B}/C).$$
\end{hypbis}

Par la description des types dans la théorie des paires propres de
corps algébriquement clos et dans les corps séparablement clos (ou
plus généralement PAC bornés) de degré d'imperfection infini,
l'hypothèse suivante est vérifiée~:

\begin{hypbis}{5}\label{H:ext_star}

  \'Etant donné un autre modèle $F'$ de $T$ contenant une
  sous-structure parfaite $C'$ au-dessus d'une sous-structure commune
  $A\subset F\cap F'$ parfaite dans $F$ et $F'$, si $\tp_0(C'/A)$ est
  stationnaire, alors il existe une sous-structure parfaite $C$ de $F$
  contenant $A$ qui est $\LL$-isomorphe à $C'$ au-dessus de $A$.
\end{hypbis}

Avec ces hypothèses, on obtient le  résultat suivant~:

\begin{theorem}
  Si la théorie $T$ satisfait les hypothèses $(\ref{H:dcl})$,
  $(\ref{H:acl})$, \textnormal{(\ref{H:qftp_star})}, $(\ref{H:full})$
  et \textnormal{(\ref{H:ext_star})} par rapport à $T_0$, alors le
  {\bf Théorème de l'indépendance} est vérifié par la
  $\star$-indépendance dès que les propriétés de {\bf Monotonie}, {\bf
    Transitivité} et {\bf Extension} sont valables pour $\indstar$
  au-dessus de parties algébriquement closes.

De plus, si $T_0$ est une théorie fortement minimale, alors $T$
satisfait le  {\bf Théorème de l'indépendance} par rapport à la
$\star$-indépendance sur tout ensemble algébriquement clos.
\end{theorem}
\begin{proof}

Soit donc une sous-structure  élémentaire  $N$ de
$F$,  deux parties algébriquement closes $A$ et $B$ la contenant et
deux parties algébriquement closes $C$ et
$D$ qui ont même type sur $N$ et vérifient
$$ C\indstar_N A \text{\qquad et \qquad}  D\indstar_N B.$$

\noindent Par {\bf Extension}, nous pouvons supposer que $\acl(D,B)$
est $\star$-indépendant de $\acl(A,B)$ sur $B$. \noindent Par
monotonie et transitivité, les sous-structures algébriquement closes
$A$, $B$ et $D$ sont $\star$-indépendantes sur $N$ (dans le sens
précédent). En particulier,

$$ \acl(AB) \indi 0_B \acl(BD),$$

\noindent et la sous-structure $E$ engendrée par $\acl(AB)$ et
$\acl(BD)$ est parfaite et coïncide avec la $\LL_0$-sous-structure
engendrée par $\acl(AB)$ et $\acl(BD)$.

On vérifie comme auparavant que $\tp_0(E/\sscl{AD})$ est
stationnaire, par l'hypothèse $(\ref{H:dcl})$ et la démonstration de
la proposition \ref{P:coher}, puisque $$\acl(A,D)=\acl_0(A,D)\cap F\, ,\,
\acl(B,D)=\acl_0(B,D)\cap F \text{ et } \acl(A,B)=\acl_0(A,B)\cap F.$$

Puisque

$$ A \indstar_N C \text{\qquad et
  \qquad} A \indstar_N D,$$ alors
$\sscl{AD}\models \qftp(\sscl{AC}/N)$, par l'hypothèse
\textnormal{(\ref{H:qftp_star})}.  Cet $\LL$-isomorphisme partiel
induit un $N$-isomorphisme entre $F$ et un autre modèle suffisamment
saturé $F_1$ de $T$ tel que $\acl(AC)$ est $\LL$-isomorphe avec
$\acl_0(AD)\cap F_1$, avec $\tp_{F_1}(D/A) = \tp_F(C/A)$.

Ainsi, à l'intérieur du modèle suffisamment saturé $F_1$, on
obtient grâce à l'hypothèse \textnormal{(\ref{H:ext_star})} une copie $E_1$
de $E$ sur $\sscl{AD}$. Par la remarque
\ref{R:rel_acl}, les copies de $\acl(AB)$ et  $\acl(BD)$ dans $E_1$
sont algébriquement closes dans $F_1$. Ainsi, la copie $B_1$ de $B$
correspondante  vérifie~:
$$ \tp_{F}(B/D) = \tp_{F_1}(B_1/D) \text{ et }  \tp_{F}(B/A) =
\tp_{F_1}(B_1/A). $$

\noindent Par saturation, il existe $B'D'$ dans $F$ tels que
$\tp_{F}(B'D'/A) = \tp_{F_1}(B_1 D/A)$.  Puisque
$\tp_F(B'/A) = \tp_{F_1}(B_1/A)= \tp_{F}(B/A)$, on peut supposer que
$B=B'$, ce qui permet de conclure.
\end{proof}

\end{document}